\documentclass[draft]{publ2}
\usepackage{amsthm,latexsym,amssymb}
\usepackage{enumerate}
\usepackage{hyperref}
\usepackage{pgf}
\usepackage[centertags]{amsmath}
\usepackage{mathrsfs}
\usepackage{graphicx}

\usepackage{units}
\numberwithin{equation}{section}

\newtheorem{thmx}{Theorem}

\newtheorem{thm2}{Theorem}
\newtheorem{cor}[thm2]{Corollary}

\newcommand{\co}{\mathcal{O}}
\newcommand{\cm}{\mathcal{M}}
\newcommand{\ch}{\mathcal{H}}
\newcommand{\p}{\mathfrak{p}}
\newcommand{\x}{\bold{x}}
\newcommand{\e}{\textrm{e}}
\newcommand{\Q}{\mathbb{Q}}
\newcommand{\Z}{\mathbb{Z}}

\newtheorem{lemma}{Lemma}

\newtheorem{proposition}[lemma]{Proposition}

\author[K\'alm\'an Gy\H ory]{{K\'ALM\'AN GY\H ORY (Debrecen)}}

\title[Bounds for $S$-unit and decomposable form equations]{Bounds for the solutions of $S$-unit equations and decomposable form equations II.}

\address{K\'{a}lm\'{a}n Gy\H ory\\
	Institute of Mathematics\\ 
	University of Debrecen\\ 
	H-4002 Debrecen, P.O. Box 400\\ 
	Hungary}

\email{gyory@science.unideb.hu}

\dedicatory{To the memory of Professor Alan Baker}

\submitted{\today}
\thanks{The author has been supported by the OTKA grants 115479 and 128088.}

\begin{document}

\begin{abstract}
In this paper we improve upon in terms of  S  the best known effective upper bounds for the solutions of S-unit equations and decomposable form equations.
\end{abstract}

\maketitle

\vspace{0.3cm}
\section{Introduction}\label{sec:1}

The $S$-unit equations
\begin{align}\label{eq:1a}
\alpha x+\beta y=1\ \textrm{in}\ x,y\in\co^\ast_S\tag{1.a}
\end{align}
(and their equivalent homogeneous versions) play a very important role in Diophantine number theory (for results and references, see e.g. the books and survey papers \cite{16}, \cite{28}, \cite{9}, \cite{19}, \cite{29}, \cite{7}, \cite{8}). Here, $\alpha,\beta$ are non-zero elements of a number field $K$, $S$ is a finite set of places on $K$ containing the infinite places, and $\co_S,\co^\ast_S$ denote the ring of $S$-integers and the group of $S$-units in $K$. The first explicit upper bounds for the heights of the solutions of equations \eqref{eq:1a} (or their homogeneous versions) were proved in our papers \cite{11}, \cite{13}, \cite{14} by means of Baker's theory of logarithmic forms. We applied these results systematically to get effective finiteness theorems in quantitative form among others to discriminant equations, power integral bases, arithmetic graphs, irreducible polynomials and decomposable form equations of the shape
\begin{align}
F(x_1,\ldots,x_m)=\delta\ \textrm{in}\ x_1,\ldots,x_m\in\co_S,\tag{2.a}
\end{align}
where $\delta\in\co_S\setminus\{0\}$, and $F(X_1,\ldots,X_m)$ is a decomposable form in $m\ge 2$ variables with coefficients in $\co_S$, whose linear factors over $\overline{K}$ have some connectedness properties; for such applications, see e.g. \cite{11}, \cite{12}, \cite{15}--\cite{18}.

Later, several improvements and further applications have been established. The most important improvements concerning \eqref{eq:1a} and \eqref{eq:2a} were obtained by Bugeaud and Gy\H ory \cite{5}, Bugeaud \cite{4}, Gy\H ory and Yu \cite{23}, and quite recently, for \eqref{eq:1a}, by Le Fourn \cite{25}. Before \cite{25} the best known bounds for the heights of the solutions of \eqref{eq:1a} and \eqref{eq:2a} were due to Gy\H ory and Yu \cite{23}; see Theorems \ref{thm:a} and \ref{thm:b} below for equation \eqref{eq:1a}, and Section \ref{sec:3} for equation \eqref{eq:2a}. Further, using our Proposition \ref{prop:5} below, Theorem \ref{thm:a} was generalized with Evertse \cite{7} to equations of the form
\begin{align}
\alpha x+\beta y=1\ \textrm{in}\ x,y\in\Gamma,\tag{1.b}
\end{align}
where $\Gamma$ denotes an arbitrary finitely generated multiplicative subgroup of $K^\ast$ of positive rank. The proofs of Theorems \ref{thm:a}, \ref{thm:b} and Proposition \ref{prop:5} are based among others on Baker's theory. The generalization concerning equation \eqref{eq:1b} also has several important applications, e.g. in our joint books \cite{7} and \cite{8} with Evertse.

In the upper bounds in Theorems \ref{thm:a} and \ref{thm:b}, the parameters depending on $S$ are $s$, the cardinality of $S$, $t$, the number of prime ideals in $S$, $P_S$ the largest norm of these prime ideals, and $R_S$ the $S$-regulator of $K$. Very recently, Le Fourn \cite{25} has improved Theorem \ref{thm:a}, replacing $P_S$ in Theorem \ref{thm:a} by $P'_S$, the third largest norm of the prime ideals in $S$. He proved his Theorem \ref{thm:c} below by combining the proof of Theorem \ref{thm:a} with his variant, Proposition \ref{prop:4} of Runge's method. This improvement is of particular importance when $P'_S$ is small compared with $P_S$ or $S$ contains at most two prime ideals when, by definition, $P'_S=1$.

In our paper, we prove a similar improvement of Theorem \ref{thm:b}, combining our Proposition \ref{prop:5} with Le Fourn's Proposition \ref{prop:4}. We obtain in Theorem \ref{thm:1} and, as a consequence, in Theorem \ref{thm:3} the best upper bounds to date in terms of $S$ for the solutions of equations \eqref{eq:1a} and \eqref{eq:2a}. Further, in Theorem \ref{thm:2} we generalize Theorem \ref{thm:c} of Le Fourn to the equation \eqref{eq:1b}. Finally, as a special case of our Theorem \ref{thm:3} we present the best bound to date for the solutions of Thue equations over $\co_S$. Our results have further consequences. Some of them will be published in a separate paper.

\section{Bounds for the solutions of $S$-unit equations}\label{sec:2}
As above, let $K$ be an algebraic number field and $S$ a finite set of places on $K$ containing the set $S_\infty$ of infinite places. Denote by $\co_S$ the ring of $S$-integers, and by $\co^\ast_S$ the group of $S$-units in $K$. Let $\alpha,\beta$ be non-zero elements of $K$, and consider the $S$-unit equation
\begin{align}
\alpha x+\beta y =1\ \textrm{in}\  x,y\in\co^\ast_S.\tag{1.a}
\end{align}
For $S=S_\infty$, $\co_S$ and $\co^\ast_S$ are just the ring of integers $\co_K$ and the unit group $\co^\ast_K$ of $K$, and \eqref{eq:1a} is called a unit equation.

To derive bounds for the solutions of \eqref{eq:1a}, we shall need some further notations. Let $d$ denote the degree of $K$, $s$ the cardinality of $S$, $R_S$ the $S$-regulator of $K$ (see Section \ref{sec:4}), $\p_1,\ldots,\p_t$ the prime ideals of $\co_K$ corresponding to the finite places in $S$, and let
$$P_S=\begin{cases}
\max_{1\le i\le t} N(\p_i),\ \textrm{if}\  t\ge 1,\\
1,\ \textrm{if}\  t=0.
\end{cases}$$
We have $s=r+t+1$, where $r$ denotes the unit rank of $K$. For $S=S_\infty$, i.e. for $t=0$, $R_S$ is just $R_K$, the regulator of $K$.

For any algebraic number $\gamma$, we denote by $h(\gamma)$ the absolute logarithmic height of $\gamma$ (cf. Section \ref{sec:4}). By the height we shall always mean the absolute logarithmic height. In \eqref{eq:1a}, let
$$H=\max (h(\alpha),h(\beta),1).$$
We use the notation $\log^\ast a=\max(\log a,1)$ for $a>0$.

Improving several earlier explicit bounds on the solutions of \eqref{eq:1a}, Gy\H ory and Yu \cite{23} proved the following two theorems with slightly smaller values for $c_1(d,s),c_2(d,r)$ and $c_4(d,r,t)$.

\begin{thmx}[Gy\H ory and Yu \cite{23}, Theorem 1]\label{thm:a}
	All solutions $x,y$ of equation \eqref{eq:1a} satisfy
	\begin{align}\label{eq:2.1}
	\max(h(x),h(y))<c_1(d,s) P_S\left(1+\frac{\log^\ast R_S}{\log^\ast P_S}\right)R_S H,
	\end{align}
	where $c_1(d,s)=(16ds)^{2(s+3)}$.
\end{thmx}

We note that for $S=S_\infty$, the bound in \eqref{eq:2.1} can be replaced by
$$c_2(d,r)R_K(\log^\ast R_K)H$$
with $c_2(d,r)=(4d)^{4(r+4)}$.

The next theorem gives a better bound for the solutions in terms of $S$. Denote by $h_K$ the class number of $K$, and put
$$\mathcal{R}=\max(h_K,c_3 dR_K),$$
where $c_3=0$, $1/d$ or $29er!r\sqrt{r-1}\log d$, according as $r=0$, $1$ or $\ge 2$.

\begin{thmx}[Gy\H ory and Yu \cite{23}, Theorem 2; see also Gy\H ory \cite{20}, Theorem A]\label{thm:b}
	Let $t>0$. All solutions $x,y$ of equations \eqref{eq:1a} satisfy
	\begin{align}\label{eq:2.2}
	\max(h(x),h(y))<c_4(d,r,t)\mathcal{R}^{t+5}\frac{P_S}{\log P_S}R_SH,	
	\end{align}
	where $c_4(d,r,t)=16^{3r+4t+12}d^{5r+t+20}$.
\end{thmx}

In terms of $S$, $s^{2s}$ is the dominating factor in the bound in \eqref{eq:2.1} whenever $t> \log P_S$. The appearance of $s^{2s}$ is due to the use of Lemma \ref{lem:2} of the present paper. Observe that in \eqref{eq:2.2} the bound does not contain $s^s$ or $t^t$. Further, in the latter bound there is $1/\log P_S$ instead of $(1+\log^\ast R_S/\log P_s)$ from \eqref{eq:2.1}. These improvements in \eqref{eq:2.2} are important for certain applications, e.g. in Gy\H ory and Yu \cite{23}, Gy\H ory, Pink and Pint\' er \cite{22} and Gy\H ory \cite{20}. In the latter paper a version of the $abc$ \textit{conjecture} over number fields is proved up to a logarithmic function.

The main tool in the proofs of \eqref{eq:2.1} and \eqref{eq:2.2} is Baker's theory of logarithmic forms, more precisely some deep results of Matveev \cite{26} and Yu \cite{31} concerning linear forms in logarithms in the complex and $p$-adic cases.

Quite recently, Le Fourn \cite{25} has proved the following improvement of Theorem \ref{thm:a}. Let
$$P'_S=\begin{cases}
\textrm{the third largest value of } N(\p_i), i=1,\ldots, t,\ \textrm{if}\  t\ge 3,\\
1, \ \textrm{if}\  t\le 2. 
\end{cases}$$
\begin{thmx}[Le Fourn \cite{25}, Theorem 1.4]\label{thm:c}
	Every solution $x,y$ of equation \eqref{eq:1a} satisfies
	\begin{align}\label{eq:2.3}
	\max(h(x),h(y))<c_1 (d,s) P'_S\left(1+\frac{\log^\ast R_S}{\log^\ast P'_S}\right)R_SH,
	\end{align}
	where $c_1(d,s)$ denotes the constant occurring in Theorem \ref{thm:a}.
\end{thmx}

Clearly, for $t>0$, \eqref{eq:2.3} is an improvement of \eqref{eq:2.1}. This improvement is of particular importance if $P'_S$ is small compared with $P_S$, for example if $1\le t\le 2$ and so $P'_S=1$. However, the bound in \eqref{eq:2.3} still contains the factor $s^{2s}$. The proof of Theorem \ref{thm:c}
 combines the proof of Theorem \ref{thm:a} with a new variant of Runge's method due to Le Fourn \cite{25}.
 
 Our first result is the following.
 \begin{thm2}\label{thm:1}
 	Let $t>0$. Every solution $x,y$ of equation \eqref{eq:1a} satisfies
 	\begin{align}\label{eq:2.4}
 	\max(h(x),h(y))<c_5(d,r,s,t)\mathcal{R}^{t+4}\frac{P'_S}{\log^\ast P'_S}\left(1+\frac{\log^\ast\log P_S}{\log^\ast P'_S}\right)R_SH,
 	\end{align}
 	where $c_5(d,r,s,t)=s^5(16\e)^{3r+4t+7}d^{4r+2t+7}$.
 \end{thm2}

In terms of $S$, this gives the best upper bound to date for the solutions of equation \eqref{eq:1a}. It improves upon in terms of $S$ both Theorem \ref{thm:b} and Theorem \ref{thm:c}.

We now compare in more detail Theorem \ref{thm:1} with Theorem \ref{thm:b} and Theorem \ref{thm:c}. The factor $P_S/\log P_S$ in \eqref{eq:2.2} is improved in \eqref{eq:2.4} to
\begin{align}\label{eq:2.5}
\frac{P'_S}{\log^\ast P'_S}\left(1+\frac{\log^\ast\log P_S}{\log^\ast P'_S}\right)
\end{align}
which improvement is particularly significant when $P'_S$ is small compared with $P_S$. If e.g. $P'_S\le\log P_S$ resp. $P'_S=1$, the factor in \eqref{eq:2.5} is at most $2\log P_S/\log P'_S$ resp. $2\log^\ast\log P_S$. Otherwise, if $P'_S>\log P_S$, then the factor in \eqref{eq:2.5} does not exceed $2P'_S/\log P'_S$. Observe that the dependence on $t$ of $c_4(d,r,t)$ is slightly better than that of $c_5(d,r,s,t)$. This is due to the fact that in Gy\H ory and Yu \cite{23} the estimates of Matveev \cite{26} and Yu \cite{31} for linear forms in logarithms are applied separately, and not through the later obtained Proposition \ref{prop:5} involving both the complex and the $p$-adic cases.

In Theorem \ref{thm:c} the factor $s^{2s}$ still occurs, in contrast with Theorem \ref{thm:1}. Further, the factor $P'_S(1+\log^\ast R_S/\log^\ast P'_S)$ in \eqref{eq:2.3} is improved in Theorem \ref{thm:1} to \eqref{eq:2.5}. Indeed, in \eqref{eq:2.4} there is an extra factor $1/\log^\ast P'_S$ and, by \eqref{eq:5.11}, $\log^\ast \log P_S$ is smaller than $\log^\ast R_S +\log 5$.

Let now $\Gamma$ be a finitely generated multiplicative subgroup of $K^\ast$ of positive rank, and consider the generalization
\begin{align}\label{eq:1b}
\alpha x+\beta y =1\ \textrm{in}\  x,y\in\Gamma\tag{1.b}
\end{align}
of equation \eqref{eq:1a}, where $\alpha,\beta$ are non-zero elements of $K$. Let $S$ denote the smallest set of places of $K$ such that $S$ contains all infinite places, and $\Gamma\subseteq \co^\ast_S$, where $\co^\ast_S$ is the group of $S$-units in $K$. In Evertse and Gy\H ory \cite{7} we proved in an effective form that equation \eqref{eq:1b} has only finitely many solutions. More precisely, we showed that there exists an algorithm which, from effectively given $K,\alpha,\beta$ and a system of generators for $\Gamma/\Gamma_{\textrm{tors}}$ and $\Gamma_{\textrm{tors}}$, computes all solutions $x,y$. We recall that $K$ is said to be effectively given if the minimal polynomial over $\Z$ of a primitive element, say $\gamma$, of $K$ over $\Q$ is given. Further, an element $\delta$ of $K$ is said to be given/effectively determinable if
$$\delta =(p_0+p_1\gamma +\cdots + p_{d-1}\gamma^{d-1})/q$$
with rational integers $p_0,\ldots,p_{d-1},q$ with $\gcd(p_0,\ldots,p_{d-1},q)=1$ that are given/can be effectively computed (see e.g. Section 1.10 in Evertse and Gy\H ory \cite{7}).

We shall need the following further parameters. Let again $H=\max(h(\alpha),h(\beta),1)$, let $\{\xi_1,\ldots,\xi_m \}$ be a system of generators for $\Gamma/\Gamma_{\textrm{tors}}$ (not necessarily a basis which is important in certain applications), let
$$\theta:=h(\xi_1)\cdots h(\xi_m),$$
$s=|S|$, $\p_1,\ldots,\p_t$ the prime ideals in $S$, and let $P_S$ and $P'_S$ be as above.

In Theorem 4.1.3 of Evertse and Gy\H ory \cite{7} we derived an explicit upper bound for the heights of the solutions of \eqref{eq:1b}, which depends on $d,s,P_S,m,\theta$ and $H$. The proof is based on our Proposition \ref{prop:5}. 

Combining Proposition \ref{prop:5} with Proposition \ref{prop:4} due to Le Fourn \cite{25}, we prove the following improvement of Theorem 4.1.3 of Evertse and Gy\H ory \cite{7}.

\begin{thm2}\label{thm:2}
	Every solution $x,y$ of equation \eqref{eq:1b} satisfies 
	\begin{align}\label{eq:2.6}
	\max(h(x),h(y))<16c_6 s\frac{P'_S}{\log^\ast P'_S}\theta\max(\log(c_6 sP'_S),\log^\ast \theta) H
	\end{align}
	where
	$$c_6(d,m)=2(m+1)\log^\ast (dm)(\log^\ast d)^2(16\e d)^{3m+5}.$$
\end{thm2}

In Evertse and Gy\H ory \cite{7} this was proved in a weaker form, with $P_S$ in place of $P'_S$.

Theorem \ref{thm:2} can be regarded as a generalization of a slightly weaker version of Theorem \ref{thm:c}. Indeed, in the special case $\Gamma=\co^\ast_S$ Theorem \ref{thm:2} gives Theorem \ref{thm:c}, in $c_1(d,s)$ with an absolute constant larger than $16$, choosing in $\co^\ast_S$ a system of generators $\{\varepsilon_1,\ldots,\varepsilon_{s-1}\}$ as in Lemma \ref{lem:2}. Then the corresponding $\theta$ is at most $c_{10}R_S$ with the constant $c_{10}$ occurring in Lemma \ref{lem:2}.

The proofs of the results presented or mentioned above involve Baker's theory. We note that there are other effective methods which provide explicit bounds for the solutions of equation \eqref{eq:1a}. Bombieri developed such a method in Diophantine approximation; see Bombieri \cite{1}, Bombieri and Cohen \cite{2}, \cite{3} and Bugeaud \cite{4}. Further, Murty and Pasten, and independently von K\"anel, Matschke and Bennett elaborated another such effective method, the so-called modular method; see e.g. Murty and Pasten \cite{27} and von K\"anel \cite{24}. However, apart from some special situations, the bounds in Theorems \ref{thm:a}, \ref{thm:b} and even more in Theorems \ref{thm:c} and \ref{thm:1}, \ref{thm:2} are better in terms of $S$.

\section{Bounds for the solutions of decomposable form equations}\label{sec:3}

Keeping the notation of the preceding section, consider the decomposable form equation
\begin{align}\label{eq:2a}
F(\x)=\delta\ \textrm{in}\ \x=(x_1,\ldots,x_m)\in\co^m_S\ \textrm{with}\ \ell(\x)\ne 0\ \textrm{for}\ \ell\in\mathcal{L},\tag{2.a}
\end{align}
where $\delta\in\co_S\setminus\{0\}$, $F\in\co_S[X_1,\ldots,X_m]$ is a decomposable form of degree $n$ (i.e. $F$ factorizes into linear forms over $\overline{K}$), and $\mathcal{L}$ is a finite set of non-zero linear forms from $\overline{K}[X_1,\ldots,X_m]$. Extending the ground field $K$ if necessary, we may assume that in \eqref{eq:2a} $F$ factorizes into linear forms over $K$. These linear factors are uniquely determined over $K$ up to proportional factors from $K$. Fix such a factorization, and denote by $\mathcal{L}_0$ a maximal subset of pairwise linearly independent linear factors of $F$. To obtain effective finiteness results on equation \eqref{eq:2a}, we make some assumptions on $\mathcal{L}_0$.

We denote by $\mathcal{G}(\mathcal{L}_0)$ the graph with vertex set $\mathcal{L}_0$ in which the edges are the unordered pairs $\{\ell,\ell' \}$, where $\ell,\ell'$ are distinct elements of $\mathcal{L}_0$ with the property that there exists a third linear form $\ell''$ in $\mathcal{L}_0$ that is a $K$-linear combination of $\ell,\ell'$. If $\mathcal{L}_0$ has at least three elements and $\mathcal{G}(\mathcal{L}_0)$ is connected, $F$ is called \textit{triangularly connected}.

When $\mathcal{G}(\mathcal{L}_0)$ is not connected, let $\mathcal{L}_{0_1},\ldots,\mathcal{L}_{0_k}$ denote the vertex sets of the connected components of $\mathcal{G}(\mathcal{L}_0)$. For $k>1$, let $\ch(\mathcal{L}_{0_1},\ldots,\mathcal{L}_{0_k})$ be the graph having vertex set $\{\mathcal{L}_{0_1},\ldots,\mathcal{L}_{0_k}\}$, in which the pair $\{\mathcal{L}_{0_i},\mathcal{L}_{0_j}\}$ is an edge if there exists a non-zero linear form $\ell_{ij}$ which can be expressed simultaneously as a $K$-linear combination of the forms in $\mathcal{L}_{0_i}$ and $\mathcal{L}_{0_j}$. Here $\ell_{ij}$ can be chosen so that the total number of non-zero terms in both linear combinations is minimal. We choose for each edge $\{\mathcal{L}_{0_i},\mathcal{L}_{0_j} \}$ such an $\ell_{ij}$, and we denote by $\mathcal{L}$ the union of these $\ell_{ij}$. Obviously, $\mathcal{G}$, $\ch$ and (ii) below depend only on the system of linear factors of $F$, but not on the choice of $\mathcal{L}_0$.

Under the assumptions that
\begin{enumerate}
	\item[(i)] \textit{the set $\mathcal{L}_0$ has rank $m$,}
	\item[(ii)] \textit{either $k=1$ or $k>1$ and the graph $\ch(\mathcal{L}_{0_1},\ldots,\mathcal{L}_{0_k})$ is connected,}
\end{enumerate}
equation \eqref{eq:2a} can be reduced to a system of $S$-unit equations in two unknowns, and using effective results concerning equation \eqref{eq:1a}, one can give effective upper bounds for the heights of the solutions of \eqref{eq:2a}. Gy\H ory and Yu \cite{23} used their Theorem \ref{thm:b} above to give in terms of $S$ the best known upper bound for the solutions of \eqref{eq:2a}
\begin{align}\label{eq:3.1}
c_7^s(P_S/\log^\ast P_S)(\log^\ast Q_S)R_S,
\end{align}
provided that (i) and (ii) hold. Here $s$ denotes again the cardinality of $S$, $R_S$ is the $S$-regulator of $K$, $P_S$ the maximal norm and $Q_S$ the product of the norms of the prime ideals $\p_1,\ldots,\p_t$ in $S$ if $t>0$, $P_S=Q_S=1$ if $t=0$, i.e. $S=S_\infty$, and $c_7$ is an explicitly given number which depends on $d,r,h_K,R_K,m,n,h(\delta)$ and $H$, an upper bound for the heights of the coefficients of $F$.

Let again $P'_S$ denote the third largest value of $N(\p_i)$, $i=1,\ldots,t$, if $t\ge 3$, and let $P'_S=1$ if $t\le 2$. Theorem \ref{thm:1} enables us to improve upon the bound of Gy\H ory and Yu \cite{23} in the following form.
\begin{thm2}\label{thm:3}
	Let $F\in\co_S[X_1,\ldots,X_m]$ be a decomposable form of degree $n$ that factorizes into linear forms over $K$ and satisfies the conditions (i) and (ii). Suppose that $t>0$. Then for every solution $\x=(x_1,\ldots,x_m)$ of \eqref{eq:2a} with $\ell(\x)\ne 0$ for $\ell\in\mathcal{L}$ if $k>1$
	\begin{align}\label{eq:3.2}
	\max_{1\le i\le m} h(x_i)<c_8^s\frac{P'_S}{\log^\ast P'_S}\left(1+\frac{\log^\ast\log P_S}{\log^\ast P'_S}\right)(\log Q_S)R_S
	\end{align}
	holds, where $c_8$ is an effectively computable positive number which depends only on $d,r,h_K,R_K,m,n,h(\delta)$ and $H$.
\end{thm2}

As was seen above, the factor in \eqref{eq:2.5} is a considerable improvement of $2P_S/\log P_S$. Hence the bound in \eqref{eq:3.2} is much better in terms of $S$ than the bound \eqref{eq:3.1} of Gy\H ogy and Yu \cite{23}.

It is clear that binary forms having at least three pairwise non-proportional linear factors are triangularly connected. Further, as is known (see e.g. \cite{21}, \cite{15}, \cite{18a}), discriminant forms and index forms are also triangularly connected, and a large class of norm forms in $m$ variables satisfies the conditions (i), (ii) with $k>1$ and $\mathcal{L}=\{X_m\}$. Therefore, our Theorem \ref{thm:3} improves upon the bounds in \cite{21}, \cite{15}, \cite{18a} concerning the $S$-integer solutions of norm form, discriminant form and index form equations.

We present a consequence for the Thue equation
\begin{align}\label{eq:2b}
F(x,y)=\delta\ \textrm{in}\ x,y\in\co_S,\tag{2.b}
\end{align}
where $F\in\co_S[X,Y]$ is a binary form of degree $n\ge 3$ which factorizes into linear factors over $K$ and at least three of these factors are pairwise non-proportional. Further, let $\delta\in\co_S\setminus\{0\}$, and $H$ an upper bound for the heights of the coefficients of $F$. Then Theorem \ref{thm:3} with $m=2$, $k=1$ gives immediately the following.
\begin{cor}\label{cor:4}
	Let $t>0$. Under the above assumptions and notation, all solutions $x,y$ of equation \eqref{eq:2b} satisfy
	\begin{align}\label{eq:3.3}
	\max(h(x),h(y))<c_9^s\frac{P'_S}{\log^\ast P'_S}\left(1+\frac{\log^\ast\log P_S}{\log^\ast P'_S}\right)(\log Q_S)R_S,
	\end{align}
	where $c_9$ is an effectively computable positive number depending only on $d,r,h_K,R_K,n,h(\delta)$ and $H$.
\end{cor}

In terms of $S$, this gives the best upper bound to date for the solutions of \eqref{eq:2b}. Corollary \ref{cor:4} improves several earlier explicit results, including Corollary 3 of Gy\H ory and Yu \cite{23}.

\section{Auxiliary results}\label{sec:4}

Keeping the notation of the preceding sections, let again $K$ denote an algebraic number field with the parameters $d,R_K,h_K$ and $r$ specified above. Denote by $\cm_K$ the set of places on $K$. For every $v\in \cm_K$ we associate an absolute value $|\ .\ |$ normalized in the usual way: if $v$ is infinite and corresponds to $\sigma:K\longrightarrow\mathbb{C}$, then we put, for $\alpha\in K$, $|\alpha|_v=|\sigma(\alpha)|$ or $|\sigma{\alpha}|^2$ according as $\sigma(K)$ is contained in $\mathbb{R}$ or not; if $v$ is a finite place corresponding to the prime ideal $\p$ in $K$, then we put $|\alpha|_v=N(\p)^{-\textrm{ord}_\p (\alpha)}$ for $\alpha\in K\setminus\{0\}$, where $N(\p)=|\co_K/\p|$ is the absolute norm of $\p$, and $\textrm{ord}_\p(\alpha)$ is the exponent of $\p$ in the prime ideal factorization of $(\alpha)$. We put $|0|_v=0$ and $\textrm{ord}_\p (0)=\infty$. Further, we denote by $d_v$ the local degree of $K$ at $v$, i.e. $d_v=[K_v:\mathbb{Q}_{v_0}]$, where $v_0$ is the place on $\Q$ lying below $v$.

The absolute logarithmic height $h(\alpha)$ of $\alpha\in K$ is defined as
$$h(\alpha)=\frac{1}{d}\sum_{v\in\cm_K}\log\max(1,|\alpha|_v).$$
It depends only on $\alpha$, and not on the choice of the number field $K$ containing $\alpha$. For properties of this height, see e.g. Evertse and Gy\H ory \cite{7}.

Let $S$ be a finite subset of $\cm_K$ containing the set $S_\infty$ of infinite places, let $\co_S$ be the ring of $S$-integers and $\co^\ast_S$ the group of $S$-units in $K$. The group $\co^\ast_S$ is of rank $s-1$ where $s=|S|$. Let $\{\varepsilon_1,\ldots,\varepsilon_{s-1} \}$ denote a fundamental system of $S$-units in $K$, and let $v_1,\ldots,v_{s-1}$ be a subset of $S$. Then $R_S$, the $S$-regulator of $K$, is the absolute value of the determinant of the matrix $(\log |\varepsilon_i|_{v_j})_{i,j=1,\ldots,s-1}$. It is a positive number which is independent of the choice of $\varepsilon_1,\ldots,\varepsilon_{s-1}$ and $v_1,\ldots,v_{s-1}$. As was mentioned above, for $S=S_\infty$ $R_S$ is just the regulator of $K$.

Again, denote by $\p_1,\ldots,\p_t$ the prime ideals in $K$ which correspond to the finite places in $S$.
\begin{lemma}\label{lem:1}
	If $t>0$, then
	$$R_K\prod_{i=1}^{t}\log N(\p_i)\le R_S\le R_Kh_K\prod_{i=1}^{t}\log N(\p_i).$$
\end{lemma}
\begin{proof}
	This is Lemma 3 in Bugeaud and Gy\H ory \cite{5}.
\end{proof}
\begin{lemma}\label{lem:2}
	There exists in $K$ a fundamental system $\{\varepsilon_1,\ldots,\varepsilon_{s-1} \}$ of $S$-units such that
	$$\prod_{i=1}^{s-1}h(\varepsilon_i)\le c_{10}R_S,$$
	where $c_{10}=((s-1)!)^2 / 2^{s-2}d^{s-1}$.
\end{lemma}
\begin{proof}
	See Lemma 1 in Bugeaud and Gy\H ory \cite{5}.
\end{proof}

For $\alpha\in K\setminus\{0\}$, the fractional ideal $(\alpha)$ can be written uniquely as a product of two ideals $\mathfrak{a}_1,\mathfrak{a}_2$, where $\mathfrak{a}_1$ is composed of $\p_1,\ldots,\p_t$ and $\mathfrak{a}_2$ is relatively prime to $\p_1,\ldots,\p_t$. Then the $S$-norm of $\alpha$ is defined as $N_S(\alpha)=N(\mathfrak{a}_2)$. In other words, $N_S(\alpha)=\prod_{v\in S}|\alpha|_v$. Notice that the $S$-norm is multiplicative. Further, $\log N_S(\alpha)\le dh(\alpha)$. 

We put again
$$Q_S=N(\p_1\cdots\p_t)\ \textrm{if}\  t>0,\  Q_S=1\ \textrm{if}\  t=0.$$
\begin{lemma}\label{lem:3}
	For every $\alpha\in\co_S\setminus\{0 \}$ and for every integer $n\ge 1$ there exists $\varepsilon\in\co^\ast_S$ such that
	$$h(\varepsilon^n\alpha)\le\frac{1}{d}\log N_S(\alpha)+n(c_3R_K+\frac{h_K}{d}\log Q_S),$$
	where, as above $c_3=0$, $1$ or $29\e r!r\sqrt{r-1}\log^\ast d$, according as $r=0$, $1$ or $r\ge 2$.
\end{lemma}
\begin{proof}
	See Lemma 3 in Gy\H ory and Yu \cite{23}.
\end{proof}

For $\gamma\in K\setminus\{ 0\}$ and $v\in\cm_K$, define $h_v(\gamma):=\log^+ (1/|\gamma|_v)$. To deal with equation \eqref{eq:1a} we consider $h_v(P)$ for
$$P\in A:=\{\alpha x,\beta y,\frac{\beta y}{\alpha x}\},$$
where $x,y$ is a solution of \eqref{eq:1a}. Denote by $S'$ the subset of $S$, deprived $S$ of its two prime ideals with largest norm. For $t\le 2$, let $S'=S_\infty$.

The following result is due to Le Fourn \cite{25}. It plays an important role in the application of his method to $S$-unit equations.
\begin{proposition}\label{prop:4}
	Let $x,y\in\co_S^\ast$ be a solution of equation \eqref{eq:1a}. Then for some $v\in S'$ and $P\in A$
	$$\frac{d_v}{d}h_v(P)\ge \frac{1}{|S|}(\max(h(x),h(y))-3H)$$
	holds.
\end{proposition}
\begin{proof}
	See Lemma 4.1 and, with slightly different notations, the corresponding arguments of Section 4 in Le Fourn \cite{25}.
\end{proof}

For $v\in\cm_K$, we put $N(v):=2$ if $v$ is an infinite place, and $N(v):=N(\p)$ if $v=\p$ is a finite place, i.e. a prime ideal of $\co_K$.

Baker's theory of logarithmic forms will be used in our proofs through the following.
\begin{proposition}\label{prop:5}
	Let $\Gamma$ be a finitely generated multiplicative subgroup of $K^\ast$ of positive rank, with system of generators $\{\xi_1,\ldots,\xi_m \}$ for $\Gamma/\Gamma_{\textnormal{tors}}$. Let $\alpha\in K^\ast$, and put
	$$H:=\max(h(\alpha),1),\ \theta:=h(\xi_1)\cdots h(\xi_m).$$
	Further, let $v\in\cm_K$. Then for every $\xi\in\Gamma$ with $\alpha\xi\ne 1$ we have
	$$\log |1-\alpha\xi|_v>c_{11}\frac{N(v)}{\log N(v)}\theta H\log^\ast\left(\frac{N(v)h(\xi)}{H}\right),$$
	where $c_{11}=2\lambda(m+1)\log^\ast(dm)(\log^\ast d)^2(16\e d)^{3m+5}$ with $\lambda=12$ if $m=1$, $\lambda=1$ if $m\ge 2$.
\end{proposition}
\begin{proof}
	This is Theorem 4.2.1 in Evertse and Gy\H ory \cite{7}. Its proof is a combination of results of Matveev \cite{26} and Yu \cite{31} concerning logarithmic forms with some results, due to Evertse and Gy\H ory \cite{7}, from the geometry of numbers.
\end{proof}

\section{Proofs of the theorems}\label{sec:5}

We keep the notation of the preceding sections.
\begin{proof}[Proof of Theorem \ref{thm:1}]
	We combine Propositions \ref{prop:4} and \ref{prop:5}, and use Lemma \ref{lem:1}, \ref{lem:2}, \ref{lem:3} as well as several ideas from Gy\H ory \cite{13} and Gy\H ory and Yu \cite{23}.
	
	Let $x,y$ be a solution of the equation
	\begin{align}
	\alpha x + \beta y = 1\ \textrm{if}\  x,y\in\co^\ast_S,\tag{1.a}
	\end{align}
	where $\alpha,\beta\in K\setminus\{0\}$. Put
	$$\ch:=\max(h(x),h(y)).$$
	For $t\ge 3$, let $S'$ denote the subset of $S$, depriving $S$ of its two prime ideals with largest norm, and for $t\le 2$ let $S'=S_\infty$. Then, by Proposition \ref{prop:4},
	\begin{align}\label{eq:5.1}
	\frac{d_v}{d}h_v(P)\ge \frac{1}{|S|}(\ch-3H)
	\end{align}
	follows for some $v\in S'$ and some $P\in A=\{\alpha x,\beta y,\beta y/\alpha x\}$. We may assume that $\ch>3H$, since otherwise we are done. Thus we have $h_v(P)>0$.
	
	First suppose that $P=\alpha x$. Then \eqref{eq:5.1} implies that $0<h_v(P)=-\log |\alpha x|_v$, whence $|\alpha x|_v<1$. We infer from \eqref{eq:1a} that $|\beta y|_v\le 4$ or $|\beta y|_v=1$ according as $v$ is infinite or finite. Hence
	\begin{align}\label{eq:5.2}
	|1-(\beta y)^{h_K}|_v=|1-\beta y|_v&\cdot |1+(\beta y)+\cdots +(\beta y)^{h_K-1}|_v\le\\
	&\le c_{12}|1-\beta y|_v,\notag
	\end{align}
	where $c_{12}=4^{h_K}$ or $c_{12}=1$, according as $v$ is infinite or not. Then it follows from \eqref{eq:1a} and \eqref{eq:5.2} that
	\begin{align}\label{eq:5.3}
	h_v(P)=&-\log |\alpha x|_v=-\log |1-\beta y|_v\le -\log |1-(\beta y)^{h_K}|_v\\
	&+\log c_{12}.\notag
	\end{align}
	
	By means of Proposition \ref{prop:5} we shall now give an upper bound for the right hand side of \eqref{eq:5.3}. Since $y\in\co^\ast_S$, there are integers $u_1,\ldots,u_t$ such that the principal ideal $(y)$ can be written in the form $(y)=\p_1^{u_1}\cdots \p_t^{u_t}$. Applying Lemma \ref{lem:3} with $S=S_\infty$, it follows that there are integers $\pi_i$ in $K$ such that $(\pi_i)=\p_i^{h_K}$ and
	\begin{align}\label{eq:5.4}
	h(\pi_i)\le \frac{2}{d}\mathcal{R}\log N(\p_i),\  i=1,\ldots,t.
	\end{align}
	Further, by Lemma \ref{lem:2} there exists in $K$ a fundamental system $\{\varepsilon_1,\ldots,\varepsilon_{r} \}$ of units such that
	\begin{align}\label{eq:5.5}
	\prod_{j=1}^{r}h(\varepsilon_j)\le c_{13}R_K
	\end{align}
	with $c_{13}=d^r$. We have
	\begin{align}\label{eq:5.6}
	y^{h_K}=\zeta\varepsilon_1^{a_1}\cdots \varepsilon_{r}^{a_r}\pi_1^{u_1}\cdots\pi_t^{u_t}
	\end{align}
	with a root of unity $\zeta$ and with appropriate integers $a_1,\ldots,a_r$.
	
	First consider the case when $v$ is infinite. Denote by $\Gamma$ the  multiplicative subgroup of $K^\ast$ generated by $\varepsilon_1,\ldots,\varepsilon_{r},\pi_1,\ldots,\pi_t$ and the roots of unity in $K$. Then	
	 $\{\varepsilon_1,\ldots,\varepsilon_{r},\pi_1,\ldots,\pi_t \}$ is a system of generators for $\Gamma/\Gamma_{\textrm{tors}}$. Put
	$$\theta:=h(\varepsilon_1)\cdots h(\varepsilon_{r})h(\pi_1)\cdots h(\pi_t).$$
	We can now apply Proposition \ref{prop:5}. We suppose that $(\beta y)^{h_K}\ne 1$, since otherwise $h(y)=h(\beta)\le H$ and, from \eqref{eq:1a}, $h(x)\le 5$ would follow which proves \eqref{eq:2.4}. We have 
	$y^{h_K}\in\Gamma$. Let $\widetilde{H}:=\max(h(\beta^{h_K}),1)$. Then by Proposition \ref{prop:5} we have
	\begin{align}\label{eq:5.7}
	-\log |1-(\beta y)^{h_K}|_v<c_{14}\frac{N(v)}{\log N(v)}\theta\widetilde{H}\log^\ast\left(\frac{N(v)h(y^{h_K})}{\widetilde{H}}\right),
	\end{align}
	where $c_{14}=2s^2(16\e d)^{3(r+t)+6}$, $N(v)=2$, $\widetilde{H}\le h_K\cdot H$, and $h(y^{h_K})\le h_K\ch$. Now, if $\ch>2sh_K^2 H$, it follows from \eqref{eq:5.1}, \eqref{eq:5.3} and \eqref{eq:5.7} that
	$$\ch<c_{15}\theta\widetilde{H}\log^\ast\left(\frac{2h_K\ch}{\widetilde{H}}\right),$$
	where $c_{15}=\frac{4}{\log 2}c_{14}$. This implies that
	\begin{align}\label{eq:5.8}
	\ch<c_{16}\mathcal{R}^2\theta(\log^\ast\theta)H,
	\end{align}
	using $h_K\le \mathcal{R}$. Here $c_{16}=2c_{15}\log(2c_{15})$.
	
	In view of \eqref{eq:5.4} and \eqref{eq:5.5} we get
	\begin{align}\label{eq:5.9}
	\theta\le c_{17}\mathcal{R}^{t+1}\prod_{i=1}^{t}\log N(\p_i),
	\end{align}
	where $c_{17}=2^td^{r-t}$. This gives
	\begin{align}\label{eq:5.10}
	\log^\ast\theta\le c_{18}(\log^\ast \mathcal{R})\log^\ast\log P_S,
	\end{align}
	where $c_{18}=3ds$ and $P_S$ denotes the maximum of the norms $N(\p_i)$, $i=1,\ldots,t$. Using the fact that $1/R_K<5$ (cf. Friedman \cite{10}), we deduce from Lemma \ref{lem:1} that
	\begin{align}\label{eq:5.11}
	\prod_{i=1}^{t}\log N(\p_i)\le 5\mathcal{R}_S.
	\end{align}
	Finally, we have
	\begin{align}\label{eq:5.12}
	\log^\ast\log P_S\le \frac{P'_S}{\log^\ast P'_S}\left(1+\frac{\log^\ast\log P_S}{\log^\ast P'_S}\right).
	\end{align}
	Now \eqref{eq:5.8}, \eqref{eq:5.9}, \eqref{eq:5.10}, \eqref{eq:5.11} and \eqref{eq:5.12} give
	$$\ch<c_{19}R^{t+4}\frac{P'_S}{\log^\ast P'_S}\left(1+\frac{\log^\ast\log P_S}{\log^\ast P'_S}\right)R_SH,$$
	where $c_{19}=3c_{16}\cdot c_{17}\cdot c_{18}$. After some computations we obtain \eqref{eq:2.4}.
	
	Next consider the case when $v$ is finite. To derive better bound for $\ch$, we make the following modification in the above arguments. Suppose that $v$ corresponds to the prime ideal $\p_t=\p$. Now we have $|\beta y|_v=1$, whence $\textnormal{ord}_\p(\beta^{h_K}y^{h_K})=0$. Putting
	$$\beta':=\beta^{h_K}\pi_t^{u_t},\  y':=y^{h_K}/\pi_t^{u_t},$$
	$\beta'y'=\beta^{h_K}y^{h_K}$ holds. But, by \eqref{eq:5.6}, $\textnormal{ord}_\p(y')=0$, hence $\textnormal{ord}_\p(\beta')=0$. This yields
	$$h_K \textnormal{ord}_\p(\beta)+u_t \textnormal{ord}_\p(\pi_t)=0$$
	which gives $|u_t|\le h_K|\textnormal{ord}_\p(\beta)|$. Further,
	$$|\textnormal{ord}_\p(\beta)|\le \frac{d}{\log N(\p)}h(\beta)$$
	(see e.g. Yu \cite{30}, p. 124). Thus, together with \eqref{eq:5.4} and $h_K\le \mathcal{R}$, this implies that
	
	\begin{align}\label{eq:5.13}
	h(\beta')\le h_Kh(\beta)+|u_t|h(\pi_t)\le \mathcal{R}^2H=:H'.
	\end{align}
	
	Let $\Gamma'$ denote the multiplicative subgroup of $K^\ast$ generated by $\varepsilon_1,\ldots,\varepsilon_{r},\pi_1,\ldots,\pi_{t-1}$ and the roots of unity in $K$. In view of \eqref{eq:5.6} we have $y'\in\Gamma'$. Put now
	$$\theta':=h(\varepsilon_1)\cdots h(\varepsilon_{r})h(\pi_1)\cdots h(\pi_{t-1}).$$
	Using again Proposition \ref{prop:5}, we infer that
	\begin{align}\label{eq:5.14}
	-\log|1-(\beta y)^{h_K}|_v=-\log |1-\beta'y'|\le\\
	c_{20}\frac{N(v)}{\log N(v)}\theta'H'\log^\ast\left(\frac{N(v)h(y')}{H'}\right),\notag
	\end{align}
	where $c_{20}=2s^2(16\e d)^{3(r+t)+3}$. Here we have
	$$h(y')=h(y^{h_K}/\pi_t^{u_t})\le \mathcal{R}^2(\ch+H).$$
	Hence it follows that
	\begin{align}\label{eq:5.15}
	\frac{N(v)h(y')}{H'}\le \mathcal{R}^2N(v)\frac{\ch+H}{H'}.
	\end{align}
	Now, as in the infinite case, we deduce from \eqref{eq:5.1}, \eqref{eq:5.3} and \eqref{eq:5.14} that
	$$\ch+H\le2sc_{21}\frac{N(v)}{\log N(v)}\theta'H'\log^\ast\left(\frac{2\mathcal{R}^2N(v)(\ch+H)}{H'}\right),$$
	whence
	\begin{align}\label{eq:5.16}
	\ch\le c_{21}(\log^\ast \mathcal{R})\frac{N(v)}{\log N(v)}\theta'H'\log^\ast(N(v)\theta'),
	\end{align}
	where $c_{21}=2sc_{20}\log(4sc_{20})$.
	
	We now estimate from above the parameters occurring in \eqref{eq:5.16}. It follows from \eqref{eq:5.4}, \eqref{eq:5.5} and \eqref{eq:5.11} that
	\begin{align}\label{eq:5.17}
	\theta'\le c_{17}R^t\mathcal{R}_S/\log N(\p)
	\end{align}
	Similarly to \eqref{eq:5.10}, we have
	\begin{align}\label{eq:5.18}
	\log^\ast\theta'\le 2c_{18}(\log^\ast \mathcal{R})\log^\ast\log P_S
	\end{align}
	with the above $c_{18}$. Using \eqref{eq:5.13}, \eqref{eq:5.17}, \eqref{eq:5.18},
	$N(\p)=N(v)$, $N(v)\le P'_S$ and $N(v)/(\log N(v))^2\le P'_S/(\log P'_S)^2$, we infer from \eqref{eq:5.16} that
	$$\ch<c_{22}\mathcal{R}^{t+4}\frac{P'_S}{\log P'_S}\left(1+\frac{\log^\ast\log P_S}{\log^\ast P'_S}\right)R_SH,$$
	where $c_{22}=3c_{21}\cdot c_{17}\cdot c_{18}$. Now as in the infinite case we get \eqref{eq:2.4} after some computations.
	
	It remains the case when in \eqref{eq:5.1} $P=\beta y$ or $\beta y/\alpha x$. In the first case \eqref{eq:2.4} immediately follows by symmetry. In the second case observe that $x'=1/x$, $y'=y/x$ is a solution of the $S$-unit equation
	$$\alpha' x'+\beta' y'=1\ \textrm{in}\  x',y'\in\co^\ast_S,$$
	where $\alpha'=1/\alpha$, $\beta'=-\beta/\alpha$ and $\beta y/\alpha x=-\beta'y'$. Then the above arguments apply to this equation with $P=-\beta' y'$ and give the same upper bound \eqref{eq:2.4} for the heights of $x',y'$ with $2H$ instead of $H$. Finally, the upper bound in \eqref{eq:2.4} follows for $\max(h(x),h(y))$ with an extra factor $2$.
\end{proof}
\begin{proof}[Proof of Theorem \ref{thm:2}]
	We follow the main steps of the proof of Theorem \ref{thm:1} in simplified form, adapting them to equation \eqref{eq:1b}. The case $m=1$ being trivial, we assume that $m\ge 2$.
	
	Let $x,y$ be a solution of equation \eqref{eq:1b}. Then $x,y$ satisfy \eqref{eq:1a} where now $S$ is the smallest subset of places of $K$ which contains $S_\infty$, so that $\Gamma\subseteq\co^\ast_S$. As above, let $t$ denote the number of finite places in $S$, and let $\ch:=\max(h(x),h(y))$. 
	
	For $t\ge 3$, let again $S'$ denote the subset of $S$ depriving $S$ of its two prime ideals with largest norm, and for $t\le 2$ let $S'=S_\infty$. Then, by Proposition \ref{prop:4}, \eqref{eq:5.1} follows for some $v\in S'$ and some $P\in A=\{\alpha x,\beta y,\beta y/\alpha x \}$. We may assume again that $\ch > 3H$, when $h_v(P)>0$.
	
	First consider the case $P=\alpha x$. Then
	\begin{align}\label{eq:5.3'}
	h_v(P)=-\log |\alpha x|_v=-\log |1-\beta y|_v.\tag{5.3'}
	\end{align}
	Applying Proposition \ref{prop:5}, we obtain
	\begin{align}\label{eq:5.4'}
-\log |1-\beta y|_v\le c_{11}\frac{N(v)}{\log N(v)}\theta H\log^\ast\left(\frac{N(v)h(y)}{H}\right)\tag{5.4'}
	\end{align}
	with $c_{11}$ occurring in Proposition \ref{prop:5}. We recall that $N(v)=2$ if $v$ is infinite, and $N(v)=N(\p)$ if $v$ is finite, where $\p$ is the prime ideal corresponding to $v$. We have in both cases $N(v)\le 2P'_S$ and $N(v)/\log N(v)\le 2P'_S/\log 2P'_S$. Now it follows from \eqref{eq:5.1}, \eqref{eq:5.3'}, \eqref{eq:5.4'} and $h(y)\le\ch$ that
	$$\ch < 2c_{11} s\frac{2P'_S}{\log (2P'_S)}\theta H\log^\ast\left(\frac{2P'_S\ch}{H}\right).$$
	Finally, this gives
	$$\ch < 8c_{11} s\frac{P'_S}{\log P'_S}\theta\max(\log(c_{11} sP'_S),\log^\ast \theta)H,$$
	which proves \eqref{eq:2.6}. For $P=\beta y$ or $\beta y/\alpha x$, we can argue in the same way as in the proof of Theorem \ref{thm:1}, and \eqref{eq:2.6} follows again.
\end{proof}
\begin{proof}[Proof of Theorem \ref{thm:3} (sketch)]
	We follow the proofs of Theorem 1 of Gy\H ory \cite{18a} and Theorem 9.6.3 of Evertse and Gy\H ory \cite{7}. The latter theorem is a less explicit version of Theorem 2 of Gy\H ory and Yu \cite{23}. We shall detail only those steps from Gy\H ory \cite{18a} or Evertse and Gy\H ory \cite{7} whose arguments differ from the earlier ones and depend on the application of our Theorem \ref{thm:1}.
	
	We shall denote by $c_{23},c_{24},\ldots,c_{37}$ effectively computable positive numbers which depend at most on $d,r,h_K,R_K,m,n,h(\delta)$ and $H$. These numbers can be made explicit by using the explicit form of Theorem \ref{thm:2}.
	
	As is pointed out in Gy\H ory \cite{18a} and Evertse and Gy\H ory \cite{7}, equation \eqref{eq:2a} can be written in the form \begin{align}\label{eq:5.19}
	\ell_1(\x)\cdots\ell_n(\x)=\delta\ \textrm{in}\ \x\in\co^m_S\ \textrm{with}\ \ell(\x)\ne 0\ \textrm{for}\ \ell\in\mathcal{L}
	\end{align}
	where, up to a proportional factor, $\ell_1,\ldots,\ell_n$ is a factorization of $F$ into linear forms in $X_1,\ldots,X_m$ with coefficients in $\co_K$, the heights of the coefficients of $\ell_1,\ldots,\ell_n$ do not exceed $c_{23}$, and the new $\delta\in\co_S\setminus\{0\}$ has height $h(\delta)\le c_{24}\log Q_S$.
	
	Let now $\x=(x_1,\ldots,x_m)\in\co^m_S$ be a solution of equation \eqref{eq:5.19} with $\ell(\x)\ne 0$ for $\ell\in\mathcal{L}$ if $k>1$, and write
	\begin{align}\label{eq:5.20}
	\ell_i(\x)=\delta_i,\ i=1,\ldots,n.
	\end{align}
	Then $\delta_i$ is a divisor of $\delta$ in $\co_S$, and $\log N_S(\delta_i)\le\log N_S(\delta)\le c_{25}h(\delta)\le c_{26}$ follows. By Lemma \ref{lem:3} there is an $\varepsilon_i\in\co^\ast_S$ such that
	\begin{align}\label{eq:5.21}
	h(\delta_i/\varepsilon_i)\le c_{27}\log Q_S,\ i=1,\ldots,n.
	\end{align}
	
	Let $\mathcal{L}_0$ be a maximal subset of pairwise linearly independent linear forms in the set of new linear forms $\ell_1,\ldots,\ell_n$. Then the new $\mathcal{L}_0$ and its associated graph $\mathcal{G}(\mathcal{L}_0)$ also satisfy the assumptions (i) and (ii) of the theorem. Let $\mathcal{L}_{0_1},\ldots,\mathcal{L}_{0_k}$ denote the vertex sets of the connected components of $\mathcal{G}(\mathcal{L}_0)$. First we assume that $k=1$. If $\{\ell_i,\ell_j\}$ is an edge of $\mathcal{G}(\mathcal{L}_0)$, then $\lambda_i\ell_i+\lambda_j\ell_j+\lambda\ell=0$ for some $\ell\in\mathcal{L}_0$ and some non-zero $\lambda_i,\lambda_j,\lambda$ in $K$ with heights not exceeding $c_{28}$. Together with \eqref{eq:5.21} this leads to an $S$-unit equation
	\begin{align}\label{eq:5.22}
	\tau_i\varepsilon_i+\tau_j\varepsilon_j+\tau\varepsilon=0\ \textrm{in}\ \varepsilon_i,\varepsilon_j,\varepsilon\in\co^\ast_S,
	\end{align}
	where $\tau_i,\tau_j,\tau$ are non-zero elements of $K$ with heights $\le c_{29}\log Q_S$. We apply now Theorem \ref{thm:1} to equation \eqref{eq:5.22} and we infer that
	$$\max(h(\varepsilon_i/\varepsilon),h(\varepsilon_j/\varepsilon))\le c_{30}^s\frac{P'_S}{\log^\ast P'_S}\left(1+\frac{\log^\ast\log P_S}{\log^\ast P'_S}\right)(\log Q_S)R_S=:A,$$
	and so, by \eqref{eq:5.21}
	\begin{align}\label{eq:5.23}
	\max(h(\delta_i/\varepsilon)),h(\delta_j/\varepsilon))\le c_{31}A.
	\end{align}
	If now $\{\ell_j,\ell_q\}$ is an edge in $\mathcal{G}(\mathcal{L}_0)$ then we deduce in the same way that there is an $\varepsilon'\in\co^\ast_S$ such that
	$$\max(h(\delta_j/\varepsilon'),h(\delta_q/\varepsilon'))\le c_{31}A.$$
	Together with \eqref{eq:5.23} this implies $h(\varepsilon'/\varepsilon)\le 2c_{31}A$, whence $h(\delta_q/\varepsilon)\le 3c_{31}A$. Using the assumption that $\mathcal{G}(\mathcal{L}_0)$ is connected and repeating the above procedure with the shortest path connecting two vertices, we infer that $h(\delta_i/\varepsilon)\le c_{32}A$ for each $i$ with $\ell_i\in\mathcal{L}_0$. But then it follows from \eqref{eq:5.19} that $h(\delta/\varepsilon^n)\le c_{33}A$. Hence $h(\varepsilon)\le c_{34}A$, and so $h(\delta_i)\le c_{35}A$ for $i=1,\ldots,n$. Regarding \eqref{eq:5.20} as a system of linear equations in $\x=(x_1,\ldots,x_m)$ and using the assumption (i), we deduce that
	\begin{align}\label{eq:5.24}
	h(x_i)\le c_{36}A\ \textrm{for}\ i=1,\ldots,n.
	\end{align}
	
	Next condsider the case $k>1$ when, by assumption (ii), the graph $\ch(\mathcal{L}_{0_1},\ldots,\mathcal{L}_{0_k})$ is connected. Then, repeating the arguments of Gy\H ory \cite{18a} or Evertse and Gy\H ory \cite{7}, we can infer as in the case $k=1$ that \eqref{eq:5.24} holds with a $c_{37}$ in place of $c_{36}$ for $i=1,\ldots,n$, whence \eqref{eq:3.2} follows.
 \end{proof}

\section*{\textbf{Acknowledgements}}
I thank Samuel Le Fourn for sending me his paper Le Fourn \cite{25} prior to its publication, and for his comments/remarks concerning our present paper.

\end{document}